\documentclass[a4paper,12pt]{amsart}

\usepackage{amsfonts}
\usepackage{amsmath}
\usepackage{amsthm}
\usepackage{amssymb}

\theoremstyle{definition} \newtheorem{defn}{Definition}[section]
\theoremstyle{plain} \newtheorem{thm}[defn]{Theorem}
\theoremstyle{plain} \newtheorem{propn}[defn]{Proposition}
\theoremstyle{plain} \newtheorem{lemma}[defn]{Lemma}
\theoremstyle{plain} \newtheorem{cor}[defn]{Corollary}
\theoremstyle{plain} 

\theoremstyle{definition}  
\theoremstyle{definition}  
\theoremstyle{remark}  
\theoremstyle{remark}  
\theoremstyle{remark}  

\theoremstyle{plain} \newtheorem*{thm*}{Theorem}
\theoremstyle{plain} \newtheorem*{cor*}{Corollary}

\newcommand {\Q} {\mathbb{Q}}

\newcommand {\Z} {\mathbb{Z}}

\newcommand {\CC} {\mathbb{C}}

\newcommand{\C}{\mathbb{C}}
\renewcommand {\epsilon}{\varepsilon}

\newcommand{\cc}{\textbf{c}}
\author{Gareth Boxall} 
\thanks{This work is based on the research supported in part by the National Research Foundation of South Africa (Grant Number 96234).}
\address{gboxall@sun.ac.za}
\address{Mathematics Division, Department of Mathematical Sciences, 
Stellenbosch University, Matieland 7602, South Africa}
\author{Gareth Jones}
\address{gareth.jones-3@manchester.ac.uk}
\address{School of Mathematics, University of Manchester, Oxford Road, Manchester, M13 9PL, UK.} 
\author{Harry Schmidt}
\address{harry.Schmidt@manchester.ac.uk}
\address{School of Mathematics, University of Manchester, Oxford Road, Manchester, M13 9PL, UK.}
\thanks{All three authors thank the Engineering and Physical Sciences Research Council for support under grant  EP/N007956/1.}
\title[Rational values and arithmetic dynamics]{Rational values of transcendental functions and arithmetic dynamics}

\date{\today}
\begin{document}
\maketitle
 
\begin{abstract}
We count algebraic points of bounded height and degree on the graphs of certain functions analytic on the unit disk, obtaining a bound which is polynomial in the degree and in the logarithm of the multiplicative height. We combine this work with p-adic methods to obtain, for each positive $\varepsilon$, an upper bound of the form $cD^{3n/4 + \varepsilon n}$ on the number of irreducible factors of $P^{\circ n}(X)-P^{\circ n}(\alpha)$ over $K$, where $K$ is a number field, $P$ is a polynomial of degree $D\geq 2$ over $K$, $P^{\circ n}$ is the $n$-th iterate of $P$, $\alpha$ is a point in $K$ for which  $\{P^{\circ n}(\alpha):n\in\mathbb{N}\}$ is infinite and $c$ depends effectively on $P, \alpha, [K:\mathbb{Q}]$ and $\varepsilon$.
\end{abstract}

\section{Introduction} We will give improvements, in certain cases, of the Bombieri-Pila Theorem on rational values of transcendental analytic functions. Our improvements apply to certain classes of analytic functions on the disk, by which we shall always mean the open unit disk in the complex plane. For $z$ in the disk, we let $\phi(z)=1/(1-|z|)$. We consider classes of functions defined by growth conditions involving $\phi$. For our first result we take functions decaying exponentially with respect to $\phi$. More precisely, we suppose that $f$ is analytic on the unit disk and that there is a subset $S$ of the unit disk containing $(0,1)$ such that, for some positive real numbers $a,b$ with $b>1$, we have 
\[
|f(z)| \le \frac{a}{b^{\phi(z)}}
\]
for $z\in S$.
\begin{thm}\label{decay} Suppose that $f$ is analytic on the unit disk with $|f(z)|\le 1$ there and that there are positive real numbers $a,b$ with $b>1$ such that $|f(z)|< ab^{-\phi(x)}$ for $z$ in $S$. There is a positive real  constant $c_{a,b}$, depending only on $a,b$ and effectively computable from them, with the following property. For $d\ge 2$ and $H\ge e$ the number of
algebraic $\alpha \in S$ such that $f(\alpha)$ is algebraic and non-zero with $[\Q(\alpha,f(\alpha)):\Q]\leq d$ and $H(\alpha),H(f(\alpha))\leq H$ is at most
\begin{align*}
c_{a,b}d^9(\log d)^2(\log H)^9.
\end{align*}
Moreover if, for some $0<r<1$, we only consider $\alpha$ such that $|\alpha|\le r$ then the bound can be improved to $c_{a,b,r}d^4(\log d)^2(\log H)^4$, where $c_{a,b,r}$ is effectively computable from $a, b$ and $r$. 
\end{thm}
Here $H(\alpha)$ is the multiplicative height of $\alpha$. See for instance Chapter 3 of Waldschmidt's book \cite{WaldschmidtBook} for a discussion of heights. Note that the rapid decay of $f$ rules out the possibility that $f$ is algebraic, unless $f=0$, and indeed such a bound may fail for algebraic functions.\\

There have been many results of this type since the fundamental work of Bombieri and Pila \cite{BombieriPila}. Building on this work, Pila proved \cite{PilaGeometricPostulation} a bound of the form $c_\epsilon H^\epsilon$ for a transcendental analytic function $f$ on $[0,1]$. This bound is essentially optimal, as Surroca \cite{SurrocaShort,SurrocaLong} and Bombieri and Pila \cite{PilaSubanalyticSurface1} showed. Surroca also showed that, in the situation of Pila's result, a bound of the form $c(\log H)^2$ holds at some unbounded sequence of heights (see also Gasbarri \cite{Gasbarri}). But to have bounds of this form holding at all heights we must make further assumptions on $f$. One possibility is to assume that $f$ satisfies some nice form of differential equation. In this direction, see work by Pila  \cite{PilaMildCurve,PilaPfaffCurve}, Thomas and the second author \cite{JonesThomasSurfaces,JonesThomasZeta}, Binyamini and Novikov \cite{BinyaminiNovikov} and Comte and Miller \cite{ComteMiller}.  In a different direction, Masser \cite{MasserZeta} proved a very precise zero estimate for the Riemann zeta function and used it to show that a $c(\log H)^2(\log\log H)^{-2}$ bound holds for the restriction of the zeta function to $(2,3)$. Besson \cite{Besson} established a similar bound for the restriction of the gamma function to a bounded interval (and also computed an explicit constant for Masser's result \cite{BessonExplicitConstant}). Independently, the first two authors proved \cite{BoxallJones1} bounds of the form $c(\log H)^{3+\epsilon}$ for gamma on the interval $(0,\infty)$, and for suitable restrictions of entire functions satisfying certain growth conditions, and further results for entire functions satisfying fairly weak growth conditions in \cite{BoxallJones2}. Related results for analytic functions on the disk are proved in the recent thesis of Pierre Villemot \cite{villemot}. 

We will prove a similar result to Theorem \ref{decay} for functions which grow exponentially along $S$, but before discussing that we mention an application of Theorem \ref{decay} to arithmetic dynamics. For this, suppose that $P$ is a polynomial in one variable, of degree $D$ at least $2$, with coefficients in a number field $K$. We write $P^{\circ n}$ for the $n$-th iterate of $P$, so $P^{\circ (n+1)}=P\circ P^{\circ n}$ and $P^{\circ 0}(z)=z$. Fix some $\alpha$ in $K$. For $n\ge 1$ consider the set
\[
S_{\alpha ,n} = \{\beta: P^{\circ n}(\beta) = P^{\circ n}(\alpha)\}.
\]
The field generated by this set is a Galois extension of $K$ since it is the splitting field of a polynomial. So we have a sequence of Galois extensions $K_n = K(S_{\alpha,n})/K$. For a generic choice of $P$ the size of this Galois group is expected to grow exponentially with $D^n$ . Boston and Jones give a precise statement concerning transcendental $\alpha$ and degree $D=2$ where they characterize the exceptional cases \cite[Theorem 4.1]{BostonJones}. In the algebraic case they give a conjecture. The general reasoning is that the Galois group of the splitting field of $P^n(X) - \alpha$ should have large image in the automorphism group of a $D$-ary tree with $n$ levels. In our situation one might expect the Galois group of $K_n$ to have large image in a $D$-ary tree of level $n-1$ and so, for $D=2$, to be of size roughly $2^{2^{n-1}}$. For $\alpha$ lying in a number field, the conjecture of Boston and Jones seems to be wide open.

For non-preperiodic $\alpha$, it follows from the conjecture of Boston and Jones that the solutions of $P^{\circ n}(X) = \alpha$ should have degree roughly $D^n$, or equivalently the polynomial $P^{\circ n}(X) - \alpha$ should have a bounded number of irreducible components. In contrast to the doubly-exponential growth of the Galois group this should hold for all polynomials. For example, for $P(X) = X^D$  such a bound follows from the irreducibility of cyclotomic polynomials. Here there is also progress for general polynomials, for example by Ingram \cite[Corollary 3]{Ingram} and de Marco \textit{et al} \cite{deMarco}. However in both  works $\alpha$ has to satisfy certain $p$-adic conditions. \\

Applying Theorem \ref{decay} with a strategy of the third author from  \cite{Harry}, and using $p$-adic methods of Ingram and DeMarco \textit{et al}  \cite{Ingram,deMarco}, we prove a result for our equation $P^{\circ n} (X) = P^{\circ n}(\alpha)$ and all algebraic $\alpha$. Moreover our results are effective. 

\begin{thm}\label{general_intro} Suppose that $P$, $K$ and $\alpha$ are as above. Let $\epsilon>0$. Then there are $c>0$, depending only on $P$ and $\epsilon$, and $c'>0$ depending only on $P$ such that 
\begin{align*}
[\Q(\beta):\Q] \ge c \frac{\min\{1,\hat{h}_P(\alpha)\}}{[K:\Q](1 + h(\alpha))^{4+\epsilon}} D^{\frac{n}{4}-\epsilon n}
\end{align*}
for
\[
 n \ge c' [K:\Q]h(\alpha) 
\]and some $\beta \in S_{\alpha,n}$. Moreover, $c$ and $c'$ are effectively computable. 
\end{thm}
Here $h$ is the logarithmic height and $\hat{h}_P$ is the canonical dynamical height associated with $P$. We note that the constants do not depend on $K$. 

The main new ingredient in our proof of Theorem \ref{general_intro} lies in the following.

\begin{thm}\label{lowerbounddynamics} Given $P$ as above, there is an effectively computable $R>0$ depending only on $P$ with the following property.  For each $\epsilon >0$ there exists an effectively computable positive constant $c_\epsilon$ depending only on $\epsilon$ and $P$ such that if $\alpha$ is algebraic with $|\alpha|\ge R$ then
\begin{align*}
[\Q(\beta):\Q]\geq c_\epsilon(1+h(\alpha))^{-1 -\epsilon} D^{\frac{n}{4}-\epsilon n}.
\end{align*}
for some $\beta \in S_{\alpha,n}$. 
\end{thm}

Note that since $|\alpha|$ is large, $\alpha$ cannot be preperiodic. 

Although it might be desirable to have such a lower bound for the degree of each element of $S_{\alpha,n}$, this cannot hold as we could have for instance $\alpha \in \Q$ and $\beta=\alpha$. Instead one could ask about `primitive' solutions $\beta$ which do not satisfy $P^{\circ m}(\beta)=P^{\circ m}(\alpha)$ for $m<n$. We don't prove a bound for these, but we are able to show that the proportion of numbers in $S_{\alpha,n}$ of low degree decays exponentially. We fix $R$ as in Theorem \ref{lowerbounddynamics} and $\alpha$ algebraic with $|\alpha|\ge R$.  For $\delta>0$ let
\[
p_{\delta,\alpha,n} = \frac{ \# \{ \beta \in S_{\alpha,n} : [\Q(\beta):\Q]\le D^{\delta n}\}}{D^n}.
\]
Then we have the following.
\begin{thm}\label{proportiondynamics}There exists an effectively computable positive constant $c_\varepsilon$ depending only on $\epsilon,P$ and $\delta$ such that  
\begin{align*}
p_{\delta,\alpha,n} \leq c_\varepsilon D^{(4\delta +\epsilon-1)n}(1+h(\alpha))^{4+\epsilon}.
\end{align*}
\end{thm}

Combining Theorem \ref{proportiondynamics} with $p$-adic methods we get a non-trivial bound on the number of irreducible factors of $P^n(X) - P^n(\alpha)$, whenever $\alpha$ is algebraic and not preperiodic. 
\begin{thm}\label{irreducible} Suppose that $\alpha$ is algebraic and not preperiodic. There exist effectively computable constants $c_1,c_2$ depending only on $\epsilon$ and $P$ such that the number $r_{\alpha,n}$ of irreducible factors of $P^n(X) - P^n(\alpha)$ over $K$ is bounded by 
\begin{align*}
r_{\alpha,n} \leq c_1 \frac{[K:\Q]^5(1+h(\alpha))^{4+\epsilon}}{\min\{1,\hat{h}_P(\alpha)\}}D^{\frac {3n}4 + \epsilon n}.
\end{align*}
for $n \geq c_2(1+h(\alpha))[K:\Q]$.
\end{thm}

To prove Theorem \ref{lowerbounddynamics} we use the theory of B\"ottcher maps in combination with the counting result, Theorem \ref{decay}. It will be clear from the proof that the same result holds for any rational function with a superattracting fixed point at infinity of order $D\ge 2$. This strategy of combining the B\"ottcher map with counting results also opens up the possibility of applying the Pila-Zannier strategy (see for instance \cite{ZannierBook}) to certain problems in arithmetic dynamics. We will pursue this elsewhere.\\

The proof of Theorem \ref{decay} develops the work of the first two authors in \cite{BoxallJones1}, but the proof is more involved as there is less room for large disks on the unit disk. The methods are quite flexible and by combining them with an idea from \cite{BoxallJones2} we are able to give a similar result for functions which grow exponentially with respect to $\phi$. We prove the following. 
\begin{thm}\label{growth_thm} Suppose that $f$ is analytic in the unit disk and that there are positive reals $a,b,c_0$ with $b>1$ such that $|f(z)|\le \phi(z)^{c_0\phi(z)}$ for $|z|\geq 1/2$ and $|f(x)|\ge ab^{\phi(x)}$ for $x$ in $(0,1)$. Then there is a positive real number $c$ with the following property. For $H\ge e$ the number of rationals $q$ in $(0,1)$ such that $f(q)$ is also rational and both $q$ and $f(q)$ have multiplicative height at most $H$ is at most
\[
c(\log H)^{18}.
\]
Here, $c$ depends only on $a,b$ and $c_0$ and can be computed from them.
\end{thm}
In fact we give a more precise form for points of bounded degree, with polynomial dependence on the degree, and with growth on a possibly larger set $S$ containing $(0,1)$.

In common with almost all counting results of Bombieri-Pila type, our proof requires the construction of a polynomial vanishing at the relevant points and satisfying various other properties. This construction is carried out in the next section. Sections \ref{decay_section} and \ref{growth_section} then give the proofs of Theorems \ref{decay} and \ref{growth_thm} respectively. In Section \ref{dynamics_real} we prove Theorems \ref{lowerbounddynamics} and \ref{proportiondynamics} and then in Section \ref{dynamics_p_adic} we combine Theorem \ref{lowerbounddynamics} with the p-adic methods mentioned above to prove Theorem \ref{general_intro}. In a short final section we give some further examples of functions to which our counting results apply. \\

We are grateful to Patrick Ingram for pointing out the B\"ottcher map to us, to Tom Tucker for his suggestion of combining non-archimedean and archimedean methods and to Hung Bui for help with Lemma \ref{powerlemma}.

\section{Polynomials for functions on the disk}\label{polynomials}
Recall that for $z$ in the unit disk we let $\phi(z)=\frac{1}{1-|z|}$. In this section we use Masser's polynomial construction \cite[Proposition 2, page 2039]{MasserZeta} to capture algebraic points of bounded height and degree on algebraic curves. First we recall the height we will be using, so fix an algebraic number $\alpha$ with minimal polynomial $P$ over the integers, and suppose that $P$ has positive leading coefficient, $a$ say. The height of $\alpha$ is defined to be
\[
H(\alpha)= \left( a \prod \max \{1, |z|\}\right)^{\frac{1}{\deg P}}
\]
with the product taken over the roots $z$ of $P$. Below we will also be using $H$ as a bound on the heights of the points considered. This shouldn't lead to confusion. Here is the result of Masser's that we need.
\begin{propn}\label{propn2} For any integer $d\ge 1$ and real $A>0, Z>0, M>0,H\ge 1,T\ge \sqrt{8d}$, let $f_1,f_2$ be analytic functions on an open set containing the closed disk of radius $2Z$ around the origin and suppose $|f_1(z)|,|f_2(z)|\le M$ on this disk. Let $\mathcal{Z}$ be a finite set of complex numbers and suppose that, for all $z,z'\in \mathcal{Z}$,
\begin{itemize}
\item[(a)] $|z|\le Z$,
\item[(b)] $|z-z'|\le \frac{1}{A}$,
\item[(c)] $[\Q(f_1(z),f_2(z)):\Q]\le d$ and
\item[(d)] $H(f_1(z)),H(f_2(z))\le H$.
\end{itemize}
If
\begin{equation}\label{AZetc}
(AZ)^T > (4T)^{96d^2/T}(M+1)^{16d}H^{48d^2}
\end{equation}
then there is a nonzero polynomial $P$ of degree at most $T$ such that $P(f_1(z),f_2(z))=0$ for all $z$ in $\mathcal{Z}$. Moreover $P$ can be taken to have integer coefficients of modulus at most $2(T+1)^2H^T$.
\end{propn}
\begin{proof} Without the `moreover' statement, this is Proposition 2 on page 2039 of Masser's paper \cite{MasserZeta} (except that there $T$ was assumed integral, but this plays no role in the proof). The `moreover' statement follows from Masser's proof, see the discussion on page 1145 of \cite{BoxallJones1}.
\end{proof}

We will need the following simple lemma as in \cite{additive}. 
\begin{lemma}\label{cover} Let $\mathcal{D}_R$ be a closed disk of radius $R$. Then $\mathcal{D}_R$ can be covered by $k$ closed disks of radius $r>0$, where
\begin{align*}
k\leq 4\left(1+\frac{\sqrt{2}}{2}\frac{R}{r}\right)^2.
\end{align*} 
\end{lemma}

We now suppose that $f$ is analytic on the disk and that there is a $c_0>0$ such that
\begin{equation}\label{diskorder1}
|f(z)| \le \phi(z)^{c_0\phi(z)}
\end{equation}
for $|z|\ge 1/2$.

Given positive reals $a,b$, with $a\ge b>1$, we let $l_{a,b}= \frac{\log a}{ \log b}$. For the relevance of this number see Lemma \ref{diskdecaynottoolarge}.

\begin{lemma}\label{diskpolynomial} Let $d\ge 2$ and $H\ge e^e$. There is a constant  $c_1>0$ depending only on  $c_0$ and effectively computable from it with the following property. If $a$ and $b$ are positive reals with $a\ge b^e>1$ and $a\ge e$ then there is a nonzero polynomial $P$ of degree at most 
\[
T=c_1l_{a,b}^3\log l_{a,b}d^4\log d(\log H)^3\log\log H
\]
 with integer coefficients of modulus at most $2(T+1)^2 H^T$ such that if $\alpha$ satisfies 
\begin{align*}
&(i) |\alpha|\leq 1- \frac{1}{2l_{a,b}d\log H}\\
&(ii) [\Q(\alpha,f(\alpha)):\Q]\leq d, \text{ and } H(\alpha), H(f(\alpha))\leq H
\end{align*}
then $P(\alpha, f(\alpha))=0$. Moreover if $|f|\leq c_0$ on the disk then we can choose $c_1$ effectively in terms of $c_0$ such that $ T=c_1l_{a,b}^2d^4(\log H)^3$. 
\end{lemma}
\begin{proof} We write $c$ for various positive constants, all effectively computable from $c_0$. Let $\mathcal{D}$ be the closed disk centred at 0 and of radius $1-\frac{1}{2l_{a,b}d\log H}$.
By Lemma \ref{cover} we can choose $N$ disks of radius $\frac{1}{32l_{a,b}d\log H}$ that cover $\mathcal{D}$ and such that 
\[
N < c\left(l_{a,b}d\log H\right)^2
\] 
with $c$ absolute. Let $a_1,\ldots, a_N$ be the centres of these disks. We may assume $a_1,\ldots, a_N\in\mathcal{D}$. For $n=1,\ldots,N$ we define the translates $f_{1,n}(z) =z +a_n, f_{2,n}(z)=f(z+a_n)$. We will apply Proposition \ref{propn2} to each of these pairs of functions.

%
Let $A=16l_{a,b}d \log H$ and $Z=\frac{2}{A}$. Let $\mathcal{Z}_n$ be the set of $\alpha-a_n$ such that:
\begin{enumerate}
\item[(i)] $|\alpha-a_n|\le \frac{1}{2A}$,
\item[(ii)] $[\Q(\alpha,f(\alpha)):\Q]\le d$ and $H(\alpha),H(f(\alpha))\le H$.
\end{enumerate}
We now apply Proposition \ref{propn2} with $\mathcal{Z}=\mathcal{Z}_n$ and $f_1=f_{1,n},f_2=f_{2,n}$. Clearly we have requirements (c) and (d). And (a) and (b) are fine, by our choices of $A,Z$. For $|z|\le 2Z$ we have
\[
|f_{2,n}(z)| \le \sup \left\{ |f(w)| : |w| \le 1-\frac{1}{4l_{a,b}d\log H} \right\}.
\]
So by \eqref{diskorder1} we can take
\begin{equation}\label{valueforM}
M=(l_{a,b}d\log H)^{cl_{a,b}d\log H}.
\end{equation}

To ensure that \eqref{AZetc} holds, we can take 
\[
T'=c l_{a,b} \log l_{a,b}d^2\log d\log H\log\log H.
\]
Then the proposition gives a nonzero polynomial $P_n$ of degree at most $T'$ with integer coefficients of modulus at most $2(T'+1)^2H^{T'}$ such that $P_n(f_{1,n}(z),f_{2,n}(z))=0$ for $z$ in $\mathcal{Z}_n$. Let $P$ be the product of all these $P_n$. Then since $N<c\left(l_{a,b}d\log H\right)^2$ the polynomial $P$ has degree at most 
 \[
 T =cl_{a,b}^3\log l_{a,b}d^4\log d(\log H)^3\log\log H
 \]
and if $\alpha$ satisfies (i) and (ii) in the statement of the lemma then $P(\alpha,f(\alpha))=0$. To complete the proof, we need to estimate the coefficients of $P$. To this end, we write
\begin{eqnarray*}
P_n(X,Y) &=&\sum_{i+j\le T'}a^{(n)}_{i,j}X^iY^j \\
P(X,Y)&= &\sum_{\alpha+\beta \le T}b_{\alpha,\beta}X^{\alpha}Y^{\beta}.
\end{eqnarray*}
And then for any $\alpha,\beta$ with $\alpha+\beta\le T$ we have
\begin{equation}\label{yuck}
b_{\alpha,\beta}=\sum a^{(1)}_{i^{(1)},j^{(1)}}\cdots a^{(N)}_{i^{(N)},j^{(N)}}
\end{equation}
with the summation taken over all sequences $i^{(1)},j^{(1)},\ldots,i^{(N)},j^{(N)}$ such that
\begin{eqnarray*}
 \alpha = i^{(1)}+\cdots + i^{(N)} \\
 \beta = j^{(1)}+\cdots +j^{(N)}\\
\end{eqnarray*}
and
\[
 i^{(n)}+j^{(n)} \le T'
\]
for $n\le N$. For these $n$ we certainly have $i^{(n)},j^{(n)}\le T'$ and so there are at most $T'+1$ choices for each of $i^{(n)},j^{(n)}$ and then there are at most $(T'+1)^{2N}$ summands in \eqref{yuck}. Each of these summands is a product of $N$ integers each bounded in modulus by $2(T'+1)^2H^{T'}$ and so each summand is bounded by $2^N(T'+1)^{2N}H^{T'N}$ and we can conclude that
$$
|b_{\alpha,\beta} |\le 2^N (T'+1)^{4N} H^{T'N}.
$$
After perhaps rechoosing the constant in $T$, we have $$2^N (T'+1)^{4N} H^{T'N}\le  2(T+1)^2H^T.$$

If we in fact have $|f(z)| \le c_0$ rather than just \eqref{diskorder1} then we proceed exactly as above, but take $M=\max \{ 1,c_0\}$ in place of \eqref{valueforM} above. We then find that for \eqref{AZetc} we can take $T'=cd^2\log H$ and then the proof continues as above, and leads to $T=cl_{a,b}^2d^4(\log H)^3$.
\end{proof}

\section{Functions on the disk with decay}\label{decay_section}
Recall that for $z$ in the unit disk, $\phi(z)=\frac{1}{1-|z|}$. Suppose that $f$ is analytic on the disk and that there are positive real $a$ and $b>1$ such that

\begin{equation}\label{diskrapiddecay}
|f(z)|\le \frac{a}{b^{\phi(z)}}
\end{equation}
for $z$ in $S \subset D$, where $(0,1)\subset S$.   We assume (as we clearly can) that $a\ge b^e>1$ and $a\ge e$. Further that $|f|\leq 1$. In this setting, we prove the following.

\begin{thm}\label{s3decay} Let $d\ge 2$ and $H\ge e^e$. There is an absolute effectively computable constant $c>0$ such that the number of algebraic $\alpha \in S$ such that $f(\alpha)$ is algebraic and non-zero with $[\Q(\alpha,f(\alpha)):\Q]\leq d$ and $H(\alpha),H(f(\alpha))\leq H$ is at most
\begin{align*}
cl_{a,b}^6(\log l_{a,b})^3d^9(\log d)^2(\log H)^9.
\end{align*}

\end{thm}
This clearly implies Theorem \ref{decay} (apart from the ``moreover'' part which we shall deduce at the end of this section). We begin with a bound on the size of the $\alpha$ in the theorem. It is here that it is crucial that $f(\alpha)$ is not zero. Recall that $l_{a,b}= \frac{\log a}{\log b}$.

\begin{lemma}\label{diskdecaynottoolarge}
Let $\alpha \in S$ be algebraic and such that $f(\alpha)$ is algebraic and non-zero, $[\Q(\alpha,f(\alpha)):\Q]\leq d$ and $H(\alpha), H(f(\alpha))\leq H$.

Then \[|\alpha|\leq 1-\frac{1}{2l_{a,b}d\log H}.\]
\end{lemma}
\begin{proof}
By a fundamental property of the height we have $|\beta|\geq H(\beta)^{-d}$ whenever $\beta\neq 0$ and $[\Q(\beta):\Q]\le d$. Thus 
\begin{align*}
H^{-d} \leq |f(\alpha)| \leq ab^{-1/(1-|\alpha|)}
\end{align*}
and so
\begin{align*}
-d\log H \leq \log a  -\frac{1}{1-|\alpha|}\log b.
\end{align*}
Hence $$1-|\alpha| \geq \log b/( d\log H +\log a) \geq \frac{\log b}{\log a}/2d\log H.$$
\end{proof}
%
As in the earlier work of the first two authors \cite{BoxallJones1,BoxallJones2} we use the following result to count zeros of functions.
\begin{propn}\label{jenson} Let $r,R$ be real numbers with $R>r>0$ and suppose $g$ is analytic on the disk $\{ |z| \le R\}$ and bounded in modulus by some $M$. If $|g(0)|\ne 0$ then the number of zeros of $g$ in $\{ |z|\le r\}$ is at most
$$
\frac{1}{\log \frac{R}{r}}\left( \log M -\log |g(0)|\right).
$$
\end{propn}
(For the proof see \cite[Theorem 1.1, page 340]{LangComplexAnalysis} or page 171 of \cite{Titchmarsh}.)

Let $\theta(z)=\frac{1+z}{1-z}$, then $\theta^{-1}(z)=\frac{z-1}{z+1}$. We need some properties of this transformation for the proof of the next theorem. 
 
A Moebius transformation maps orthogonal circles to orthogonal circles. In particular, since $\theta$ maps the interval $[-1,1]$ to the real interval $[0,+\infty]$, it maps each circle of radius $0<  r<1$ centred at 0 to a circle $C_r$ that intersects the real line orthogonally in the points $\theta(-r)<1, \theta(r)>1$. Note that $\theta$ sends the unit circle to the imaginary line. Denote by $U_r $ the closed disk that has $C_r$ as its boundary and note that $C_{r'}$ is contained in $U_r$ for $r'\leq r$. Thus $U_r$ is the image of the disk of radius $r$ centred at 0 under $\theta$. We need the following to continue.
\begin{lemma} \label{moebius} Let $0< r\le r'<1$ and $s\geq 1$ be such that $\theta(-r') \leq \theta(-r)/s$. Then $U_r/s$ is contained in $U_{r'}$.  
\end{lemma} 
\begin{proof} We note that $\tilde{\theta}(z)=\theta(z)/s$ is again a Moebius transformation that sends $[-1,1]$ to $[0,+\infty]$. It is enough to show that $\tilde{\theta}$ sends the circle of radius $r$ around 0 into $U_{r'}$. For this it is enough to note that the image under $\tilde{\theta}$ of this circle intersects the real line orthogonally at the points $\theta(-r)/s\geq \theta(-r')$ and $\theta(r)/s\leq \theta(r')$. 
\end{proof}

We can now prove Theorem \ref{s3decay}. For this we assume $|f(z)|\leq 1$ on the unit disk. All the constants can be effectively computed. We will write $c$ for various positive constants. Recall that $d\ge 2$ and $H\ge e^e$.  By Lemma \ref{diskdecaynottoolarge} the $\alpha$ we wish to count satisfy $|\alpha| \leq 1-1/2l_{a,b}d\log H$. So we can apply Lemma \ref{diskpolynomial} to find an absolute constant $c_1>0$ such that there is a nonzero polynomial $P(X,Y)$ of degree at most 
\[
T=c_1l_{a,b}^2d^4(\log H)^3
\] 
with integer coefficients satisfying $|P|\le 2(T+1)^2H^T$ (where $|P|$ is the maximum modulus of the coefficients of $P$) and such that $P(\alpha,f(\alpha))=0$ for all the $\alpha$'s we're counting. Fix such a $P$.

Note that for $x>1$ we have $\phi(\theta^{-1}(x))= \frac{1}{2}(x+1)$ and so for these $x$ we have
\begin{equation}\label{fofthetasmall}
|f(\theta^{-1}(x))|\le \frac{a}{b^{\frac{1}{2}(x+1)}}.
\end{equation}
Let $R(X)=P(X,0)$ and $Q=P-R$. We can assume that $R$ is nonzero since if $R$ is the zero polynomial then $P$ is divisible by $Y$ and we can divide by $Y$ until the corresponding $R$ is non-zero, and this doesn't affect us since we're only counting $\alpha$'s for which $f(\alpha)\ne 0$. Note that $Q$ is divisible by $Y$. We now seek a point $x_0$ such that $|P(\theta^{-1}(x_0),f(\theta^{-1}(x_0)))|$ is not too small (and is, in particular, non-zero). We consider $Q$ and $R$ separately, and first find a large interval on which $|Q(\theta^{-1}(x),f(\theta^{-1}(x)))|$ is very small. More precisely, we want an interval of length at least $T+1$ on which
\begin{equation}\label{Qsmall}
|Q(\theta^{-1}(x),f(\theta^{-1}(x)))|\le \frac{1}{2(x+1)^T}.
\end{equation}
Since $Y$ divides $Q$ there is a polynomial $Q'$ of degree at most $T$ and with $|Q'|\le |Q|$ such that $Q=YQ'$. We have
\begin{eqnarray*}
|Q'(\theta^{-1}(x),f(\theta^{-1}(x)))| &\le &\sum |Q'||\theta^{-1}(x)|^i|f(\theta^{-1}(x))|^j \\
&\le &(T+1)^2|Q|
\end{eqnarray*}
since $|\theta^{-1}(x)|$ and $|f(\theta^{-1}(x))|$ are both at most $1$. Since $Q=YQ'$ we find that
\begin{eqnarray*}
|Q(\theta^{-1}(x),f(\theta^{-1}(x)))| &\le & |f(\theta^{-1}(x))| (T+1)^2|Q|\\
& \le & \frac{a}{b^{\frac{1}{2}(x+1)}} (T+1)^2 |Q|
\end{eqnarray*}
where we have used \eqref{fofthetasmall} for the second inequality. We know that $|Q|\le |P| \le 2(T+1)^2 H^T$ and so
$$
\log |Q| \le c T \log H.
$$
Combining the last two estimates shows that for \eqref{Qsmall} it suffices to have
$$
\log 2 +T\log (x+1) +\log a +2\log (T+1)+cT \log H \le \frac{1}{2}(\log b )(x+1).
$$
For this it is sufficient to ensure that
$$
l_{a,b}^{-1}(x+1) \ge \max \{ cT \log (x+1), cT\log H\}.
$$
So we can choose $x$ around $c l_{a,b}T\log T\log H/\log \log H $ 
 and so we can take
$$
[cl_{a,b}T \log T\log H/\log \log H,2c l_{a,b}T\log T\log H/\log \log H]
$$
as our interval.

We now consider $R$. The function
$$
(X+1)^{\lfloor T\rfloor}R(\theta^{-1}(X))
$$
is a polynomial in $X$ of degree at most $T$ with integer coefficients (we write $\lfloor T\rfloor$ for the integer part of $T$). Since our interval has length greater than $T+1$ there is an integer $m$ inside it such that $(m+1)^{\lfloor T\rfloor}R(\theta^{-1}(m))$ is not zero. As $R$ has integer coefficients this number is also an integer and we have
$$
|R(\theta^{-1}(m))| \ge \frac{1}{(m+1)^T}.
$$
But $m$ is in the interval on which \eqref{Qsmall} holds and so we find that
\begin{equation}\label{propertyofm}
|P(\theta^{-1}(m),f(\theta^{-1}(m)))| \ge \frac{1}{2(m+1)^T}.
\end{equation}
Now suppose that $\alpha $ in $S$ is algebraic and such that $f(\alpha)$ is also algebraic and non-zero, and $[\Q(\alpha,f(\alpha)):\Q]\le d$ and $H(\alpha),H(f(\alpha))\le H$. By Lemma \ref{diskdecaynottoolarge} we have
\begin{equation}\label{upperboundonq}
|\alpha|\le 1-\frac{
1}{2l_{a,b}d\log H}.
\end{equation}

We define the function 
$$g(z)=P(\theta^{-1}(m\theta (z)),f(\theta^{-1}(m\theta (z)))).
$$
Then $\theta^{-1}(\theta(\alpha)/m)$ is a zero of $g$ and we want to bound the absolute value of $\theta^{-1}(\theta(\alpha)/m)$. To this end we find $0\leq r<1$ such that $r\ge |\alpha|$ and $\theta(-r)\leq \theta(-|\alpha|)/m$. By (\ref{upperboundonq}) and the fact that $$m\leq 2cl_{a,b}T\log T\log H/\log \log H$$ it is sufficient to find an $r$ such that $$1-r\leq c\log \log H/(l_{a,b}^2dT\log T(\log H)^2).$$ So we pick 
\[
r=1-\frac{c \log \log H}{l_{a,b}^2dT\log T(\log H)^2}
\] 
and, by Lemma \ref{moebius}, $|\theta^{-1}(\theta(\alpha)/m)|\leq r $. So it suffices to count the zeros of $g$ in a disk of radius $r$. Then let 

\[
R=1-\frac{1-r}{2}.
\]


Now we can use some standard estimates to deduce that
  \[
 R/r \ge \exp\left(\frac{c\log \log H}{l_{a,b}^2dT\log T(\log H)^2}\right).
 \]
For example as $\exp(t)\le 1+2t$ for $0\leq t\leq 1$ we deduce that 
\begin{eqnarray*}
R/r =& 1+\frac{1-r}{2r} \\
\geq & 1+ 2\frac{1-r}4\\
\geq &\exp(\frac{1-r}4)\\
=& \exp\left(\dfrac{c\log \log H}{l_{a,b}^2dT\log T(\log H)^2}\right).
\end{eqnarray*} 
Thus 
\begin{align*}
1/\log(R/r)\leq cl_{a,b}^2dT\log T(\log H)^2/\log \log H.
\end{align*}

 We also bound $-\log|g(0)|$. From (\ref{propertyofm}) it follows that 
\begin{align*}
-\log |g(0)| \leq c\log l_{a,b}T\log T.
\end{align*}
Finally we bound $M=\max\{|g(z)|:|z|\leq R\}$. Recall that $|P|\le 2(T+1)^2H^T$ and that, by assumption, $|f(z)|\le 1$ for all $z$ in the disk. So
\begin{eqnarray*}
\log M \le \log |P| +2\log(T+1) \le c T\log H.
\end{eqnarray*}
That implies 
\begin{align*}
\log M - \log |g(0)|\leq c \log l_{a,b} T\log T\log H/\log\log H.
\end{align*}
Now with Proposition \ref{jenson} we deduce that number $N$ of $\alpha$'s that we're counting satisfies
 \begin{align*}
 N\leq c l_{a,b}^2\log l_{a,b} d(T\log T)^2(\log H)^3/(\log \log H)^2
 \end{align*}
 and so 
\begin{align*}
N \leq cl_{a,b}^6(\log l_{a,b})^3d^9(\log d)^2(\log H)^9.
\end{align*}

%
%


In an application that we give later, we only need count points $(\alpha,f(\alpha))$ for $\alpha$ in a compact subset of the disk. In this case, we can improve the bound. 

To begin, we need only take $T=cd^2\log H$, where $c$ now depends on the compact subset. We proceed as above, but take $$r=1-\frac{c \log \log H}{l_{a,b}T \log T\log H},$$ with $c$ again depending on the compact set. We then find a final bound of
$$
cl_{a,b}\log l_{a,b}d^4(\log d)^2 \left( \log H\right)^4.
$$

\section{Functions with growth}\label{growth_section}
We now prove Theorem \ref{growth_thm}. Large parts of the proof are very similar to the previous proof and we leave these parts for the reader. Suppose that $f$ is analytic on the unit disk and that there are positive real numbers $a,c_0$ and $b>1$ such that \eqref{diskorder1} holds for $|z|\geq 1/2$ and such that for $x\in S$
\[
|f(x)| \ge a b^{\phi(x)}.
\]

As before, $S$ is a subset of the unit disk containing the interval $(0,1)$ and $l_{a,b}=\frac{\log a}{\log b}$. We assume $a\ge \max\{b^e,e\}$. (In the case where this extra assumption is not satisfied, one could obtain the conclusion of Theorem \ref{growth_thm} by applying the following result to a function obtained from $f$ by scaling.) 

\begin{thm}\label{s4growth}
Let $d\ge 2$ and $H\ge e^e$. There is a positive constant $c$ effectively computable from $c_0$ such that the number of algebraic $\alpha\in S$ such that $f(\alpha)$ is algebraic with $[\mathbb{Q}(\alpha,f(\alpha)):\mathbb{Q}]\leq d$ and $H(\alpha),H(f(\alpha))\leq H$ is at most 
\begin{align*}
cl^{17}_{a,b}(\log l_{a,b})^9d^{18}(\log d)^9(\log H)^{17}(\log\log H)^6
\end{align*} 

\end{thm}

First we record the analogue in this setting of Lemma \ref{diskdecaynottoolarge}. The proof is very similar and is left for the reader. 

\begin{lemma}\label{diskgrowthnottoolarge} Let $\alpha \in S$ be algebraic and such that $f(\alpha)$ is algebraic, $[\Q(\alpha,f(\alpha)):\Q]\leq d$ and $H(\alpha), H(f(\alpha))\leq H$.

Then \[|\alpha|\leq 1-\frac{1}{2l_{a,b}d\log H}.\]
\end{lemma}
We can now start the proof of Theorem \ref{s4growth}. We write $c$ for various positive constants which are effectively computable from $c_0$. Applying Lemma \ref{diskpolynomial} gives us a nonzero polynomial $P$ of degree at most $$T=cl_{a,b}^3 \log l_{a,b}d^4\log d(\log H)^3\log\log H$$ with integer coefficients and with $|P|\le 2(T+1)^2H^T$ such that $$P(\alpha,f(\alpha))=0$$ for the $\alpha$'s we're interested in. Let $\theta$ and $\theta^{-1}$ be as in the previous proof. For real $x>1$ we have $\phi(\theta^{-1}(x))=\frac{1}{2}(x+1)$ and so
\begin{equation}\label{lowerboundfofthetainverse}
|f(\theta^{-1}(x))| \geq ab^{\frac{1}{2}(x+1)}.
\end{equation}
Let $L\le T$ be the degree of $P$ in $Y$, and write
\[
P(X,Y)= \sum_{i=0}^L P_i(X)Y^i
\]
where $P_0,\ldots,P_L$ are polynomials in $X$ with integer coefficients satisfying $|P_i|\le |P|$. Let
\[
P'(X,Y) = \sum_{i=0}^L P_{L-i}(X)Y^i.
\]
By \eqref{lowerboundfofthetainverse} $f(\theta^{-1}(x))$ is nonzero for $x>1$ and so
\begin{equation}\label{GROWTHpintermsofoneoverf}
P(\theta^{-1}(x),f(\theta^{-1}(x)))= f(\theta^{-1}(x))^L   P'\left(\theta^{-1}(x), \frac{1}{f(\theta^{-1}(x))}\right).
\end{equation}
We will use this to find an $x$ which isn't too large and is such that $P(\theta^{-1}(x),f(\theta^{-1}(x)))$ isn't too small, much as we did in the previous proof. Let $R(X)= P'(X,0) (=P_L(X))$ and $Q(X,Y)=P'(X,Y)-R(X)$. Note that $R(X)$ is not the zero polynomial, by the definition of $L$. For now, suppose that $Q$ is also not the zero polynomial. Since $Y$ divides $Q$ we can argue exactly as we did after \eqref{Qsmall} to show that
\begin{equation}\label{Qsmallingrowthcase}
\left| Q\left(\theta^{-1}(x),\frac{1}{f(\theta^{-1}(x))}\right)\right| \le \frac{1}{2(x+1)^T}
\end{equation}
for $x$ in the interval
\[
[cl_{a,b}T\log T\log H/\log \log H,2cl_{a,b}T\log T\log H/\log \log H].
\]

We can then proceed as before to find an $m$ in this interval such that
\[
|R(\theta^{-1}(m))| \ge \frac{1}{(m+1)^T}
\]
so that
\[
\left|P'\left(\theta^{-1}(m), \frac{1}{f(\theta^{-1}(m))}\right)\right| \ge \frac{1}{2(m+1)^T}.
\]
Then by \eqref{GROWTHpintermsofoneoverf} we also have
\[
|P(\theta^{-1}(m),f(\theta^{-1}(m)))|\ge \frac{1}{2(m+1)^T}.
\]
In fact, $|P(\theta^{-1}(m),f(\theta^{-1}(m)))|$ is a bit larger than this, but we don't seem to be able to exploit this. Anyway, we can now proceed as in the previous proof. We see that it suffices to count the zeros of
\[
g(z)= P(\theta^{-1}( m\theta(z)), f(\theta^{-1}( m\theta(z))))
\]
in a disk of radius 
\[
r= 1-\frac{c\log \log H}{l^2_{a,b}dT\log T(\log H)^2}.
\] 
We estimate this as before, except that now the bound on $M$ is larger as our $f$ is now growing. Let 
\[
R=1-\frac{1-r}{2}=1-\frac{c\log \log H}{2 l^2_{a,b}dT\log T(\log H)^2}.
\]
Using reasoning from the previous section we get
\[
\frac{1}{\log(R/r)}\leq cl^2_{a,b}dT\log T (\log H)^2/\log \log H.
\]

Before estimating $M= \max\{ |g(z)|:|z|\le R\}$ we first estimate $M'= \max \{ |f(\theta^{-1}( m\theta(z)))| : |z|\le R\}$. For $|z|\le R$ we have
\[
\phi(\theta^{-1}( m\theta(z))) \le \frac{1}{2}+\frac{m(1+R)}{2(1-R)}\le c l^3_{a,b}d T^2(\log T)^2(\log H)^3/(\log \log H)^2
\]
and so by \eqref{diskorder1} we have
\[
M' \le (l^3_{a,b}d T^2(\log T)^2(\log H)^3/(\log \log H)^2)^{c l^3_{a,b}d T^2(\log T)^2(\log H)^3/(\log \log H)^2}
\]
and as
\[
M\le (T+1)^2 |P| M'^T
\]
we have
\[
\log M \le c l^3_{a,b}\log{l_{a,b}}d\log{d}T^3(\log T)^3(\log H)^3/\log\log H
\]
After estimating $-\log|g(0)|$ as in the previous proof, we use Proposition \ref{jenson} to get a final bound of 
\[
cl_{a,b}^5\log l_{a,b} d^2\log dT^4(\log T)^4 (\log H)^5/(\log \log H)^2
\]
\[
\le cl^{17}_{a,b}(\log l_{a,b})^9d^{18}(\log d)^9(\log H)^{17}(\log\log H)^6
\]

But what if $Q$ is the zero polynomial? Then our original $P$ has the form $P(X,Y) = P_L(X)Y^L$. The first factor has at most $T$ zeros since it is a polynomial of degree at most $T$. The second factor is non-zero at all the points we care about. So we end up with a much better bound in this case. This completes the proof.

\section{Dynamics over $\C$} \label{dynamics_real}

In this section we work in $\C$  and assume that we have fixed an embedding of $\overline{\Q}$ in $\C$. We consider the dynamical system associated to a polynomial $P \in K[X]$ of degree $D \geq 2$ where $K$ is a number field. In what follows we may and will assume that $P$ is monic. This is because we can make $P$ monic by passing to a conjugate $\gamma^{-1} P(\gamma X)$ for $\gamma \neq 0$  lying in an extension of degree at most $D-1$ over a field of definition for $P$. The filled Julia set of $P$ is the set of $z \in \mathbb{C}$ such that $P^{\circ n}(z) \nrightarrow \infty $. 
For each such polynomial $P$ (in fact over $\C$) there exists a neighbourhood $U_\infty$ of $\infty$ contained in the complement of the filled Julia set and a bi-holomophism 
\begin{align*}
\Phi_P:U_\infty \simeq D(0,r)
\end{align*} 
for some $r > 0$, where $D(0,r)$ is an open disk of radius $r$ centred at 0, that satisfies 
\begin{align*}
z\in U_\infty\text{ implies } P(z)\in U_\infty,
\end{align*}
\begin{align*}
\Phi_P(\infty)=0
\end{align*}
\begin{align*}
\text{and }\Phi_P(P(z)) = \Phi_P(z)^D.
\end{align*} 
This is a theorem of B\"ottcher \cite[Theorem 6.7]{Milnor}. We will assume that $0$ is not contained in $U_\infty$. For example if $P(z) = z^D$ then $\Phi_{P} = 1/z$ and we can pick $U_\infty = \{z: |z| > 1\}$. In fact if all critical points of $P$ lie in its filled Julia set then we can always take $U_\infty$ to be the complement of the filled Julia  set of $P$ and $r=1$. In general we can't continue   $\Phi_P$ as a holomorphic function to the whole complement of the filled Julia set \cite[p.6-7]{Milnor}. \\

In what follows we pick $U_\infty$ such that $\Phi_P^{-1}$ can be extended to $D(0,r^{1/D})$ and such that $r \leq e^{-\pi/6}$. This choice can be made effectively and makes the computations more transparent. For any $\alpha \in U_\infty$  we can now define the function 
\begin{align*}
f^*(\tau) =\frac{1}{\Phi_P^{-1}(\exp(2\pi i (\tau - i/24) )\Phi_P(\alpha))}
\end{align*}   
on $\{\tau: \Im(\tau) >0\}$. 
This $f^*$ has the property that 
\begin{align}\label{image}
1/f^*(\{k/D^n +  i/{24}: k = \lfloor -\frac{D^n-1}2\rfloor , \dots,\lfloor \frac{D^n-1}2\rfloor  \}) = S_{\alpha ,n},
\end{align}
with $S_{\alpha, n}$ as defined on page 2, and further that 
%
\begin{align}\label{growth}
0<|f^*(\tau)| <c\exp(-2\pi \Im(\tau))
\end{align}
for some effective positive $c$ depending only on $P$.  We pull back $f^*$ to the unit disk via $\mu(z) =  i(1+z)/(1-z)$ and 
by (\ref{growth}) this pullback $f = f^*\circ \mu$ and the set 
\begin{align*}
S = \mu^{-1}\left( \{\tau: \Im(\tau) \geq 1/24, |\Re(\tau) | \leq \frac12 \} \right)
\end{align*}
satisfy the conditions of Theorem \ref{decay}  (with effectively computable $a,b$). 
Recall that to each polynomial (or even just rational) map $P$ we can associate a canonical dynamical height $\hat{h}_P:\overline{\Q} \rightarrow [0,\infty)$ defined by
\begin{align*}
\hat{h}_P(z) =\lim_{n\rightarrow \infty}\frac{h(P^{\circ n}(z))}{D^n} 
\end{align*}
with the property $\hat{h}_P(P(z)) = D\hat{h}_P(z)$ and
\begin{align*}
|\hat{h}_{P}(z)- h(z)|\leq c_{P,h}
\end{align*} 
for a constant $c_{P,h}$ depending (effectively) only on $P$ (in particular $\hat{h}_P \neq 0 $)  \cite[Theorem 3.20]{Silverman}.
\begin{proof}[Proof of Theorems \ref{lowerbounddynamics} and \ref{proportiondynamics}]
Since $\hat{h}_P(\beta) =\hat{h}_P(\alpha) $ for $\beta \in S_{\alpha, n}$ we deduce from the above  that  \[H(\mu^{-1}( k/D^n + i/24)), H(f(\mu^{-1}(k/D^n + i/24)))\leq cH(\alpha)D^{2n}\]  for $k  \in \{\lfloor -\frac{D^n-1}2\rfloor , \dots,\lfloor \frac{D^n-1}2\rfloor\}$ and from Theorem \ref{decay} that 
\begin{align*}
D^n \leq c(1+h(\alpha))^4n^4(\log D)^4(1+\log d)^2d^4
\end{align*}
where $d$ is a bound for the degree of $f(k/D^N + i/24)$ over $\Q$ and $c$ is effective. Theorem \ref{lowerbounddynamics} now follows from (\ref{image}). \\

Now pick a set $S_\delta \subset \{k/D^N+i/24: k=\lfloor -\frac{D^n-1}2\rfloor , \dots,\lfloor \frac{D^n-1}2\rfloor\}$ such that $f(\mu^{-1}(\tau))$ has degree at most $D^{\delta n}$ for $\tau \in S_\delta $. Plugging this bound in Theorem \ref{decay} for our $f$ we deduce Theorem \ref{proportiondynamics}. 
\end{proof}
Now we want to investigate how the involved constants behave if we vary $P$ in a family of polynomials. 

Let $V$  be a quasi-affine variety over a number field $K$ with coordinate functions $T= (T_1, \dots, T_m)$ on $\mathbb{A}^m$ and let  $P \in \mathcal{K}[X]$, where $\mathcal{K} = K(T)$, of degree $\deg(P)=D \geq 2$. (Here we work with $X$ instead of $z$ to distinguish it from the complex variable.)  For each $\cc \in V(\overline{\Q})$ such that the coefficients of $P$ are defined at $\cc$ we can specialize to $P_\cc \in \overline{\Q}[X]$ and associate the canonical dynamical height $\hat h_\cc =\hat{h}_{P_{\cc}}$ to $P_\cc$. We define the Weil-height on  $V$ in the usual way
\begin{align*}
 [K:\Q]h_V(\cc) =\sum_{v\in M_K}\log \max\{1, |t_1|^{n_v}_v,\dots, |t_m|^{n_v}_v\}
\end{align*}
for $(t_1, \dots, t_m) =T(\cc)$ where $K$ is a number field containing $t_1,\dots, t_m$, $M_K$ is the set of places of $K$, suitably normalized such that $|p|_v = p^{-1}$ for a prime $p$ satisfying $v|p$, and $n_v$ is the local degree $[K_v:\Q_v]$. 
We shrink $V$ if necessary such that $P_\cc$ is always defined and such that $\deg P _\cc = \deg P $ for all $\cc \in V(\overline\Q) $. 
\begin{lemma}\label{uniform123} For all $\cc \in V(\overline{\Q})$   
\begin{align*}
|\hat{h}_\cc-h|\leq \delta_1h_V(\cc) +\delta_2
\end{align*}
for $\delta_1, \delta_2$  effectively computable positive real constants depending only on $P$. 
\end{lemma}
\begin{proof} We set $bP =a_0X^D+\dots + a_D$ with
$ a_0, \dots, a_D, b \in \mathcal{O}_K[T_1,\dots, T_m]$ where $\mathcal{O}_K$ is the ring of integers of $K$, $ a_0, \dots, a_D, b$ are co-prime and we set $D_V$ to be the maximum of the degrees of $b,a_i, i=0, \dots, D$.  It is more convenient to work with $\tilde{P} = P(1/X) = A/ (bX^D)$ with $A \in \mathcal{O}_K[T,X]$. Note that $(A,b X^D)=1$ and that this holds for every specialization $P_\cc$ that we consider.  By the theory of resultants there exist polynomials $A_0,B_0, A_\infty, B_\infty \in \mathcal{O}_K[T,X]$ of degree at most $D-1$ in $X$ and at most $D_v(2D-1)$ in $T_1,...,T_m$ such that 
\begin{align*}
A_0A +bX^DB_0 =R, ~~ A_\infty A+bX^DB_\infty = RX^{2D-1} 
\end{align*}
for $R\in  \mathcal{O}_K[T_1,\dots, T_m]$. Now we can follow the proof in \cite[section 5.1] {HabeggerJonesMasser}  but we repeat some of the arguments for the reader's convenience. Let $(t_1,\dots, t_m) = T(\cc)$ be the coordinates of the specialized point, $r=R(\cc)$ and $x$ be an algebraic specialization of $X$. At a  non-archimedian place $|\cdot|$ we have that $|r|$ and $|rx^{2D-1}|$ and therefore also $|r|\max\{1,|x|^{2D-1}\}$ are bounded above by 
\begin{align*}
\max\{1,|t_1|, \dots, |t_m|\}^{D_V(2D-1)}\max\{1,|x|\}^{D-1}\max\{|A(x,t_1,...,t_m)|,|bx^D|\}.
\end{align*} 

At an archimedean place we get $|r|\max\{1,|x|^{2D-1}\}$ bounded above by
\begin{align*}
L\max\{1,|t_1|, \dots, |t_m|\}^{D_V(2D-1)}\max\{1,|x|\}^{D-1}\max\{|A(x,t_1,...,t_m)|,|bx^D|\}
\end{align*} 

where $L$ is the sum of the lengths of $A_0,B_0,A_\infty, B_\infty \in \mathcal{O}_K[X,T_1,\dots, T_m]$.  Taking the product over all places the factor $r$ cancels out and we obtain 
\begin{align*}
Dh(x) \leq D_V(2D-1)h(\cc) + h(\tilde{P}_\cc(x)) + L.
\end{align*}
The inequality $h(\tilde{P}_\cc(x)) \leq D_Vh_V(\cc) + Dh(x) + O(1)$ with effective $O(1)$ follows from straightforward estimates. Since $h(1/x) = h(x)$ we obtain 
\begin{align*}
|h(P_\cc(x)) -Dh(x)| \leq D_V(2D-1)h_V(\cc) +c
\end{align*}
with $c$ effective and depending only on $P$. Now we can use the telescope summing trick to obtain 
\begin{align*}
|\hat{h}_{\cc} - h|\leq \delta_1h_V(\cc) + \delta_2
\end{align*}
where $\delta_1 =D_V(2D-1)/(D-1)$ and $\delta_2 = c/(D-1)$. 
\end{proof}

\begin{lemma}\label{Bottcher} Let $P \in \mathcal{O}[X]$ be monic and of degree $\deg P \geq 1$ where $\mathcal{O}$ is the ring of holomorphic functions  on a compact domain $B\subset \CC^n$. For each $\cc\in B$ there exists a B\"ottcher domain $U_\cc$ around infinity for the specialized $P_\cc \in \CC[X]$ and there is $r>0$ such that $D(\infty,r) = \{z: |z| >r \}\subset U_\cc$ for all $\cc \in B$.   Moreover the B\"ottcher map gives rise to a holomorphic map $\Phi(z,\cc)$ on $D(\infty,r)\times B$ such that 
\begin{align*}
\Phi(P(z,\cc),\cc) = \Phi(z,\cc)^D.
\end{align*}
\end{lemma}
\begin{proof}
For the proof we just have to follow the proof in \cite[Theorem 6.7]{Milnor} and we repeat some of the arguments. We can write 
\begin{align*}
P = z^D(1 + O(1/|z|))
\end{align*}
and since the coefficients of $P(z)$ are uniformly bounded on $B$, the  $ O(1/|z|) $  term is bounded uniformly for $|z| > r > 0$. We set $Z$ such that $z = \exp(Z)$ and set $F(Z) = \log P(\exp(Z))$. With the right lifting of $F$ we get $F(Z) = DZ + O(\exp(-\Re(Z))) $ and we can choose $\sigma$ so large (and $r$ so big)  such that for $\Re(Z) > \sigma$,   $|F(Z) - DZ| < 1$. Now the sequence of functions $L_k = F^{\circ k}  (Z)/D^k$ converges uniformly on this domain to a function $L$ that satisfies $L \circ F = DL$. Moreover it satisfies $L(Z + 2\pi i) = L(Z) + 2\pi i$. Thus the function $\Phi(z) = \exp(-L (\log z))$ is well defined on $D(\infty,r)$ and satisfies $\Phi(P)= \Phi^D$ (Note that $|P(z)| >r$ for $|z|>r$.)  Moreover as the convergence is uniform in $B$ the statement about holomorphy follows.  
\end{proof}

\begin{thm}\label{parameter} Let $V$ be a quasi-affine variety defined over a number field $K$ and let $P \in K(V)[X]$ be a monic polynomial of degree $D\geq 2$. Further let $B \subset V(\mathbb{C})$ be a  compact set with the property  that all specializations $P_\cc$ are defined and $\deg(P_\cc) = \deg(P)$ for  $\cc \in B$. Let $D(\infty,r)$ be the associated B\"ottcher domain as in Lemma \ref{Bottcher}. Pick an algebraic $\alpha \in D(\infty,r)$ and let $S_{\alpha,n,\cc}$ be 
\begin{align*}
S_{\alpha,n,\cc} = \{\beta: P^{\circ n}_\cc(\beta) = P^{\circ n}_\cc(\alpha)\}.
\end{align*}
For every $\epsilon > 0$ there exists a constant $c_\epsilon$ depending only on $\epsilon$ and $P$ (but not $\cc$ or $\alpha$) such that 
\begin{align*}
[\Q(\beta):\Q] \geq c D^{\frac n4 - \epsilon n} (1 + h_V(\cc) + h(\alpha))^{-1 -\epsilon}.
\end{align*}
for some $\beta \in S_{\alpha,n}$.
Moreover we can bound the points of low degree in $S_{\alpha,n,\cc}$. For every $\epsilon > 0$ and $\delta > 0$ there exists a constant $c_{\epsilon}$ such that the number of points in $S_{\alpha,n,\cc}$  of degree at most $D^{\delta n}$ divided by $D^n$ is bounded above by 
\begin{align*}
 c_\varepsilon D^{(4\delta +\epsilon -1)n}(1+ h_V(\cc) +  h(\alpha))^{4 + \epsilon}.
\end{align*}
\end{thm}
\begin{proof} For the proof, given Lemma \ref{uniform123}, we only have to make sure that the constant in Theorem \ref{decay} can be chosen uniformly for $\cc \in B$. From the construction of the B\"ottcher map it follows that we only need a uniform bound on the coefficients of $P$, which we have.  
\end{proof}
We note that if we have an effective bound for the length of $P_\cc$ as $\cc$ varies over $B$ all the constants in Theorem \ref{parameter} are effective.

\section{General Galois-bounds}\label{dynamics_p_adic}

We now prove Theorem \ref{general_intro} and Theorem \ref{irreducible}. We recall the statement of Theorem \ref{general_intro}.
\begin{thm}\label{general} Suppose that $K$ is a number field and that $P\in K[X]$ has degree $D$ at least $2$. Let $\epsilon>0$. Then there exist effectively computable $c>0$, depending only on $P$ and $\epsilon$, and $c'>0$ depending only on $P$ such that if $\alpha\in K$ and $n\ge c'[K:\Q]h(\alpha)$ then
\begin{align*}
[\Q(\beta):\Q] \ge c \frac{\min\{1,\hat{h}_P(\alpha)\}}{[K:\Q](1 + h(\alpha))^{4+\epsilon}} D^{\frac{n}{4}-\epsilon n}
\end{align*}
for some $\beta \in S_{\alpha,n}$. 
\end{thm}

As before, we may assume $P$ is monic. Now  fix a number field $K$.  In what follows, given a prime $p$ we write $|\cdot |_v$ for the extension of a $p$-adic valuation on $\Q$ to $K$ with the standard normalization and $K_v$ for the completion of $K$ with respect to $v$. We write $\C_v$ for the completion of the algebraic closure of $K_v$.  And we let $D_v(0,r)$ and $D_v(\infty, r)$ denote the sets of points in $\C_v$ such that $|z|_v < r$ and $|z|_v>r$, respectively. Suppose that 
\[
P(z)= z^D + a_{1}z^{D-1} + \dots + a_D
\]
is a polynomial over $K$. For primes $p$ not dividing $D$ we define 
\begin{align*}
\delta_v = \max_{i=1, \dots, D}\{1, |a_i|_v\}. 
\end{align*}
And for primes $p$ that do divide $D$ we set 
\begin{align*}
\delta_v =\frac{\max_{i=1,\dots, D}\{1,|a_i|_v\}p^{\frac1{p-1}}}{|D|_v}.
\end{align*}
We will use the following result, due to De Marco \textit{et al} (\cite[Theorem 6.5]{deMarco}) which extends work of Ingram (\cite{Ingram}).
\begin{thm}[Ingram \cite{Ingram}, De Marco \textit{et al} \cite{deMarco}]In the setting described above, there exists an injective analytic function $\Phi_{v}$  with domain $D_v(\infty, \delta_v)$  such that $\Phi_{v}(P(z)) = \Phi_{v}(z)^D$. Moreover, $\Phi_v$ has the property that if $z$ lies in a finite extension of $K_v$ then this extension also contains $\Phi_v(z)$. 
\end{thm}

Before we proceed with the proof we need more information about cyclotomic extensions of $p$-adic fields.  First some (basic) group theory. 
\begin{lemma}\label{grouptheory} For a prime $q$ let $a>1$ be an integer coprime to $q$. If $q\neq 2$ let $e$ be the order of $a$ in $(\Z/q\Z)^*$ and $m$ be maximal with the property $a^e = 1 \mod q^m$. If $q=2$ let $e$ be the order of $a$ in $(\Z/4\Z)^*$ and $m$ be maximal such that $a^e = 1 \mod 2^m$. For $n \geq m$ the order of $a$ in $(\Z/q^n\Z)^*$ is $eq^{n-m}$. 
\end{lemma}
\begin{proof}
We prove this by induction on $n\ge m$, and prove simultaneously that $n$ is the maximal $k$ such that $a^{eq^{n-m}}=1\mod q^ k$. For $n=m$ this is our assumption, so assume that the result holds for some $n\ge m$. So
\[
a^{eq^{n-m}}=1 \mod q^n
\]
 and so 
\[
a^{eq^{n-m}}=1 +bq^n
\]
for some $b$. If $q|b$ then $a^{eq^{n-m}}=1 \mod q^{n+1}$, a contradiction. So $(b,q)=1$. We have
\[
a^{eq^{n+1-m}}=1+bq^{n+1}+\sum_{i=2}^qb^iq^{ni}\binom{q}{i}.
\]
Considering separately the cases where $q>2$ and $q=2$, we see that the sum on the right of this expression will be divisible by $q^{n+2}$, and so 
\[
a^{eq^{n+1-m}}=1 +\tilde{b}q^{n+1}
\]
for some $\tilde{b}$ with $(\tilde{b},q)=1$. Thus $n+1$ is maximal such that $a^{eq^{n+1-m}} = 1 \mod q^{n+1}$. 

The order, $f$ say, of $a$ in $(\Z/q^{n+1}\Z)^*$ divides $eq^{n+1-m}$. On the other hand $e|f$, since $a^f=1 \mod q$. So $f=eq^k$ for some $k\le n+1-m$. If $k<n+1-m$ then $k=l-m$ for some $l<n+1$, and then $a^{eq^{l-m}}=1 \mod q^{n+1}$, contradicting our inductive assumption on the maximality of $l$. So $f=eq^{n+1-m}$, as required.
\end{proof}
Now we can deduce the following. 
\begin{lemma}\label{order} Let $a$ and $ D$ be  coprime integers both greater than $1$. For a prime $q$ dividing $D$ let $m_q$ be as $m$ in Lemma \ref{grouptheory}. Let $m$ be maximal among these $m_q$. Then for $b$ dividing a power of $D$, the order of $a$ in $(\Z/b\Z)^*$  is at least $bD^{-m}$. 
\end{lemma}
\begin{proof} Let $b = \prod_{q|b} q^{n_q}$ be the prime decomposition of $b$. Suppose that  $a^{f} = 1 \mod b$. Then $a^f = 1 \mod q^{n_q}$ for all $q|b$. If $n_q \geq m_q$ then by the previous lemma, $q^{n_q-m_q}$ divides $f$. So $f \geq \prod_{q|b}q^{n_q - m}$ and we are done since $ \prod_{q|b} q^{n_q - m}\geq bD^{-m}$.
\end{proof}
For a positive integer $l$ we denote by $\zeta_l$ a primitive $l$-th root of unity. 
\begin{lemma}\label{gal1} If a prime $p$ does not divide a positive integer $D$ then for any $b$ dividing a  power of $D$
\begin{align*}
[\Q_p(\zeta_{b}):\Q_p] \geq bD^{-m}
\end{align*}
for some $m \leq \frac{D-1}{\log 2}\log p$.
\end{lemma}
\begin{proof}
By  \cite[Proposition 7.12]{Neukirch} the degree $[\Q_p(\zeta_{b}):\Q_p]$ is equal to the order of $p$ in $(\Z/b\Z)^*$. Let $e$ be the order of $p$ in  $(\Z/q\Z)^*$ where $q$ is a prime dividing $D$ and let $m$ be maximal such that $p^e=1 \mod q^m$. Then $q^m \leq p^e-1 \leq p^{D-1}$ and so $m \leq\frac{D-1}{\log 2}\log p $. We conclude with Lemma \ref{order}. 
\end{proof}

\begin{lemma}\label{gal2} Let $\mathcal{K} = \Q_p(\zeta_{b})$ where the prime $p$ does not divide $b$. Then  
\begin{align*}
[\mathcal{K}(\zeta_{p^k}):\mathcal{K}] = p^{k-1}(p-1).
\end{align*}
\end{lemma} 
\begin{proof} By \cite[Proposition 7.12]{Neukirch} the extension $\mathcal{K}$ is unramified over $\Q_p$. In particular $p$ stays prime in $\mathcal{K}$. Hence we can apply the Eisenstein criterion to the cyclotomic polynomial $\phi_{p^k}$ in $\mathcal{K}$ just as in the proof of \cite[Proposition 7.13]{Neukirch}.
\end{proof}

Combining the previous lemmas we obtain the following.
\begin{cor} \label{galcor} For a prime $p$ and a positive integer $b$ dividing a power of a positive integer $D$,
\begin{align*} 
[\Q_p(\zeta_{b}):\Q_p] \geq bD^{-m}
\end{align*}
for some $m \leq \frac{D-1}{\log 2}\log p$.
\end{cor}
\begin{proof} We write $b = p^{n_p}\tilde{b}$ where $n_p,\tilde{b}$ are positive integers such that $p$ does not divide $\tilde{b}$. Since $\Q_p(\zeta_{b}) = \Q_p(\zeta_{\tilde{b}}, \zeta_{p^{n_p}})$ we obtain the corollary from Lemmas \ref{gal1} and \ref{gal2}.
\end{proof}

Before we prove Theorem \ref{general} we prove the $p$-adic version of the Galois bounds. These turn out to have a stronger asymptotic than the complex ones but in general with  a worse dependence on the point $\alpha$. 

\begin{thm}\label{Galoispadic} Suppose that $|\alpha|_v > \delta_v$ for some $v$. Then 
\begin{align*}
[\Q(\beta): \Q] \geq \frac{D^{n-m}}{[K:\Q]}
\end{align*} 
for some  $m \leq \frac{(D-1)[K:\Q]}{\log 2} h(\alpha)$ and some $\beta \in S_{\alpha,n}$. Moreover with the same $m$ we have
\begin{align*}
\# \{\beta' \in S_{\alpha,n}: [K(\beta'):\Q] \leq d\} \leq d^2D^{2m}. 
\end{align*}
\end{thm}
\begin{proof} 
For $z \in D(\infty, \delta_v)$ we have $|P^{\circ n}(z)|_v = |z|_v^{D^n}$. From this we can deduce that $|\beta|_v = |\alpha|_v$ for $\beta \in S_{\alpha,n}$. Moreover, the  pre-images of a critical point of $P$ lie outside of $D(\infty, \delta_v)$ for all places $v$. Thus $P^{\circ n}(X)-P^{\circ n}(\alpha)$ 
has non-vanishing discriminant and so $D^n$ distinct solutions. For each $\beta \in S_{\alpha,n}$ we have $\Phi_v(\beta) = \Phi_v(\alpha)\zeta$ for a $D^n$-th root $\zeta$ of 1. Since $\Phi_v$ is injective the function $z\mapsto \Phi_v(z)/\Phi_v(\alpha)$ induces a bijection between $S_{\alpha,n}$ and the $D^n$-th roots of unity.\\
 Now fix an integer $b$ dividing $D^n$ and let $\beta \in S_{\alpha,n}$ be such that $\Phi_v(\beta)/\Phi_v(\alpha)$ is a primitive $b$-th root $\zeta_b$ of $1$. First by the properties of $\Phi_v$, 
\[
K_v(\Phi_v(\beta),\Phi_v(\alpha))\subseteq K_v(\beta).
\]
 So by our choice of $\beta$ we have $\zeta_{b}\in K_v(\beta)$ and so by Corollary \ref{galcor} we have
\begin{align}
[K_v(\beta):\Q_v] \geq bD^{-m}, \label{ineq1}
\end{align}
for some $m \le \frac{D-1}{\log 2}\log p$. By our hypothesis on $\alpha$, we have 
\[
|\alpha|_v \ge p^{\frac{1}{[\Q_p(\alpha):\Q_p]}}.
\]
So
\[
\log p \le [K:\Q]h(\alpha).
\]
Choosing $b = D^n$ gives the first part of the statement as 
\[ 
[K_v(\beta):\Q_v]\le [K(\beta):\Q] \leq [\Q(\beta):\Q][K:\Q].
\]
 For the second part we note that it follows from (\ref{ineq1}) and the fact that $\Phi_v(z)/\Phi_v(\alpha)$ is a bijection between $S_{\alpha,n}$ and the $D^n$-th roots of unity that 
 \[
 \# \{\beta' \in S_{\alpha,n}: [K(\beta'):\Q] \leq d\}
 \]
 \[
 \leq \#\{\zeta \in \overline{\Q}:  \zeta  \text{ is a primitive }b\text{-th root of }1\text{ for } b \leq dD^m\}.
 \]

Since the number of primitive b-th roots of unity is bounded  by $b$ we obtain the second part of the statement. 
\end{proof}
From this point on we also allow $v$ to be archimedean. For the archimedean places $v$ we set $\delta_v$ to be the $R$ from Theorem \ref{lowerbounddynamics}. In order to prove Theorem \ref{general}, we first observe that the set of $\alpha$ such that neither Theorem \ref{lowerbounddynamics} nor Theorem \ref{Galoispadic} apply is a set of bounded height. In the following result $K$ is, as before, a number field over which $P$ is defined. However this time we do not require $\alpha$ to lie in $K$. For any number field $\mathcal{K}$, we denote by $M_\mathcal{K}$ the set of places on $\mathcal{K}$ normalized in the usual way. 
\begin{lemma}\label{boundedheight} The set 
\[
\{ \alpha \in \overline{\Q} : |\alpha|_v\le \delta_v \text{ for all }v \in M_{K(\alpha)}\}
\]
is a set of bounded height, with the bound depending only on $P$.
\end{lemma}
\begin{proof}
Let $S=\{ v \in M_K: \delta_v>1\}$. Then 
\begin{eqnarray*}
H(\alpha)^{[K(\alpha):K]}& =&\prod_{w\in M_{K(\alpha)}} \max \{ 1, |\alpha|_w^{n_w}\}\\
&=&\prod_{v\in S}\prod_{w\in M_{K(\alpha)},w|v} \max\{ 1, |\alpha|_w^{n_w}\}\\
&\le&\prod_{v\in S}\prod_{w\in M_{K(\alpha)}, w|v} \delta_w^{n_w}
\end{eqnarray*}
where $n_w=[K(\alpha)_w:K_w]$. For each $v\in M_K$ we have
\[
\sum_{w\in M_{K(\alpha)},w|v}n_w \le [K(\alpha):K]
\]
and so
\[
H(\alpha) \le \prod_{v\in S} \delta_v
\]
and the result follows. \end{proof}

\begin{proof}[\textit{Proof of Theorem \ref{general}}]
By Lemma \ref{boundedheight} the set of $\alpha \in \overline{\Q}$ such that $|\alpha|_v\leq \delta_v$ for all places $v$ has bounded height, say by a constant $c_1$ (depending only on $P$). Recall  that $|\hat{h}_P-h|\leq c_2$ where $c_2$ only depends on $P$. If $h(\alpha) > c_1$ and so $|\alpha|_v > \delta_v$ for some $v$, we can apply Theorem \ref{lowerbounddynamics} or \ref{Galoispadic} according to whether or not $v$ is archimedean.  If $h(\alpha) \leq c_1$ and $\hat{h}_P(\alpha)>0$ then we pick the minimal positive integer $k$ such that $\hat{h}_P(P^{\circ k}(\alpha))=D^k \hat{h}_P(\alpha) >  c_1 +c_2$.  
Since $P^{\circ n}(P^{\circ k}(\alpha))= P^{\circ n}(P^{\circ k}(\beta))$ for $\beta\in S_{\alpha,n}$ we have that
$P^{\circ k}(S_{\alpha,n})\subseteq S_{P^{\circ k}(\alpha),n}$. For positive integers $m,m',l$ and $l'$ we have that $S_{P^{\circ m}(\alpha),l}\cap S_{P^{\circ m'}(\alpha),l'}$ is nonempty if and only if $m=m'$. It follows quickly that $\# S_{\alpha,n}\ge D^{n-1}$.   As $P^{\circ k}$ is generically a $D^k$ to 1 map  we have $\#P^{\circ k}(S_{\alpha,n}) \geq D^{n-k-1}$. We set $d$ to be minimal such that $[K(\beta):\Q]\leq d$ for all $\beta \in S_{\alpha,n}$. If $|P^{k}(\alpha)|_v > \delta_v$ for some non-archimedean $v$ we apply Theorem \ref{Galoispadic} and obtain
\begin{align}\label{1}
D^{n-k-1} \leq d^2D^{2m}
\end{align}
for some $m \leq \frac{(D-1)[K:\Q]}{\log 2}(c_1+2c_2)D$. If there is no such non-archimedean $v$ then $|P^{\circ k}(\alpha)|_v > \delta_v$ for an archimedean $v$ and we can apply Theorem \ref{proportiondynamics} to find that \begin{align}\label{2}
D^{n-k-1} \leq cd^{4+\epsilon}(1+h(\alpha))^{4+\epsilon}
\end{align}
for every $\epsilon >0$, with some constant $c$ depending only on $\epsilon$ and $P$. Combining (\ref{1}) and (\ref{2}) we conclude that we can find $\beta \in S_{\alpha,n}$ such that 
\begin{align*}
[K(\beta):\Q] \geq c\frac{D^{n/4 -\epsilon n}\min\{1,\hat{h}(\alpha)\}}{(1+h(\alpha))^{4+\epsilon}}
\end{align*}
for all $n\geq c'[K:\Q]h(\alpha)$, where $c$ and $c'$ are constants with $c$ depending on $P$ and $\varepsilon$ and $c'$ depending only on $P$. 
\end{proof}
Before we prove Theorem \ref{irreducible} we need a preparatory lemma. 
\begin{lemma}\label{powerlemma} For each $\theta \geq 2$ there exists a constant $c_{\theta}$ with the following property. Let $c\geq 1$ and let $d_1,\dots, d_{M}$ be positive integers such that $\sum_{i=1}^Md_i = X$ and, for all $R > 0$,
\begin{align}
\sum_{ \{i : d_i \leq R\}}d_i \leq cR^{\theta}. \label{condition}
\end{align}
 Then $M \leq c_{\theta} cX^{1- 1/\theta}$.   
\end{lemma}
\begin{proof} Fix $M \geq 1$. We are going to minimize $X$ while preserving the condition on the sub-sums. We define a sequence of integers $a_j, j=1, \dots $ as follows. Let $a_1 = \lfloor c\rfloor  $ and for $k \geq 2 $ let $a_k$ be maximal subject to the restriction  $\sum_{i = 1}^{k} ia_i \leq ck^{\theta}$. With our fixed $M$, let $m$ be maximal such that  $ \sum_{i = 1}^{m
}a_i  \leq M$ and let $n = M- \sum_{i = 1}^{m
}a_i$. We claim that $X_0 =  \sum_{i = 1}^{m}ia_i +n(m+1)$ is the minimal value for $X$ if $M$ is fixed.

Suppose there is $d_1, \dots, d_M$ subject to (\ref{condition}) such that  $X_{min} = \sum_{i=1}^{M}d_i$ is minimal and smaller than $X_0$. Clearly we can assume that the $d_i$ are non-decreasing.

We define  $b_k = \#\{i:d_i = k\}$ and note that there exists $K\leq m$ such that $b_K < a_K$. Otherwise we would have $X_{min} \geq X_0$ and this would contradict our assumtion that $X_{min}$ is smaller than $X_0$. We now suppose that  $K$ is minimal with the property $b_K< a_K$.  We can replace $d_{I}$ by $d_{I}-1$ for $I = \sum_{k=1}^Kb_k$ without  violating (\ref{condition}) thus contradicting  that $ X_{min} = \sum_{i=1}^{M}d_i$ is minimal.

So we have shown that $X_0$ is the minimum value and will now estimate it from below. Since $ck^{\theta -2} \ll_\theta a_k \ll_\theta ck^{\theta -2}$ we find that $M \ll_\theta m^{\theta-1}$ and $X_0 \gg_\theta cm^{\theta}$. So $X_0 \gg_\theta cM^{\theta/(\theta-1)}$ and we are done since for any $X$ we have $c_\theta cX_0^{1 - 1/\theta}\leq c_\theta cX^{1- 1/\theta}$. 
\end{proof}
\begin{proof}[\textit{Proof of Theorem \ref{irreducible}}] 

As in the proof of Theorem \ref{general} we choose $k$ minimal such that $|P^{\circ k}(\alpha)|_v > \delta_v$ for some place $v$. Then $D^k \ll 1/\hat{h}_P(\alpha)$. First suppose that $v$ is archimedean. With similar arguments as above we deduce from Theorem \ref{proportiondynamics} that
\[
\{\beta \in S_{\alpha,n}: [K(\beta):\Q] \leq d\} \leq c \frac{(1+h(\alpha))^{4+\epsilon} }{\min\{1,\hat{h}_P(\alpha)\}}d^{4+\epsilon}
\]
and so
\begin{align}\label{est1}
\{\beta \in S_{\alpha,n}: [K(\beta):K] \leq d\} \leq c \frac{[K:\Q]^5(1+h(\alpha))^{4+\epsilon} }{\min\{1,\hat{h}_P(\alpha)\}}d^{4+\epsilon}
\end{align}
with the constant $c$ depending only on $P$ and $\varepsilon$. 

We apply Lemma \ref{powerlemma} with the $X$ there taken to be $D^n$ and $M=r_{\alpha,n}$. We choose $d_1,\ldots,d_{r_{\alpha,n}}$to be the degrees of the irreducible factors of $P^{\circ n}(z) - P^{\circ n}(\alpha)$ over $K$ and set $\theta=4+\varepsilon$. By (\ref{est1}) condition \eqref{condition} is fulfilled, and so Lemma \ref{powerlemma} gives 
\begin{align*}
r_{\alpha,n} & \leq c_{\epsilon}\frac{[K:\Q]^5(1+h(\alpha))^{4+\epsilon} }{\min\{1,\hat{h}_P(\alpha)\}}D^{\frac34n + \epsilon n}.
\end{align*}
For $v$ non-archimedean we get from Theorem \ref{Galoispadic} and Lemma \ref{powerlemma} 
\begin{align*}
r_{\alpha,n}\leq c[K:\Q]^2D^{2m}D^{n/2}
\end{align*}
with $m\leq \frac{(D-1)[K:\Q]}{\log 2} h(\alpha)$. Combining these two inequalities yields the result. 
\end{proof}
\section{Further examples}
We conclude the paper by showing how Theorem \ref{decay} applies to various modular functions. For instance, let
\[
\lambda(\tau) = \frac{ \left( 2 \sum_{n=0}^\infty q^{\frac{1}{4}\left(2n+1\right)^2}\right)^4}{\left(1+2\sum_{n=1}^\infty q^{n^2}\right)^4}
\]
be the modular $\lambda$-function, where $q=\exp(\pi i\tau)$. See for instance chapter 7, section 7 of \cite{Chandrasekharan}. By (8.1) on page 117 of \cite{Chandrasekharan} $\lambda$ is bounded on the half plane $\Im\tau \ge 1$ and there is some $c>0$ such that 
\begin{equation}\label{lambdagrowth}
0<|\lambda (\tau)| < ce^{-\pi \Im\tau}
\end{equation}
for $\tau$ with sufficiently large imaginary part, and real part in $[-1,1]$ say. As in section \ref{dynamics_real} we construct a related function on the disk. In order to obtain a bounded function on the disk, we don't work with the whole upper-half plane, but the half-plane given by $\Im\tau>1$. The transformation $z\mapsto \frac{2i}{1-z}$ takes the unit disk to this half plane and so the function
\[
\lambda^*(z) = \lambda \left( \frac{2i}{1-z}\right)
\]
is bounded on the unit disk. Let $S$ be the union of the interval $(0,1)$ with the inverse image under the Moebius transformation above of the set $S'$ of $\tau$ with real part in $[-1,1]$ and imaginary part large enough so that \eqref{lambdagrowth} holds. Bounds on the number of algebraic points of bounded height and degree on the graph of the function $\lambda^*$ and on the graph of $\lambda$ restricted to $\Im\tau >1$ are clearly equivalent, with suitable changes in constants. And by \eqref{lambdagrowth} above the function $\lambda^*$ satisfies the hypotheses of our Theorem \ref{decay}. So we get a bound
\[
c'd^9(\log d)^2(\log H)^9
\]
for algebraic points of degree at most $d$ and height at most $H$ (with $d\geq 2$ and $H\geq e$) on the graph of $\lambda$ restricted to the set $S'$. Here $c'$ is absolute, and could in principle be computed. 

This is weak compared to Schneider's theorem \cite[Theorem 6.3, page 56]{Baker}, which implies that $\lambda(\tau)$ will be transcendental when $\tau$ is algebraic and not quadratic. But exactly the same argument applies to the derivatives $\lambda',\lambda''$ of $\lambda$, and leads to results which appear to be new, although they could perhaps also be obtained using Binyamini's result in \cite{Binyamini}. (For what is known about transcendence here see \cite{Diaz}).

Similarly, our result applies to the discriminant function
\[
\Delta(\tau) =(2\pi)^{12}q\prod_{n\ge 1} (1-q^n)^{24}
\]
where we now use $q=\exp(2\pi i\tau)$ and to other cusp forms (see for instance Theorem 8.1 on page 62 and Proposition 7.4 on page 59 of \cite{SilvermanAdv}). Indeed it applies to any modular form $\sum_{n \ge 0} c_n q^n$ with algebraic $c_0$, by applying the above method to $\sum_{n \ge 1} c_n q^n$.

\end{document}